\documentclass[preprint,12pt,authoryear]{elsarticle}

\usepackage{amssymb}
\usepackage{graphicx}
\usepackage{enumerate}
\usepackage{amsthm}
\usepackage{amsfonts}
\usepackage{amsmath}
\usepackage{enumitem}
\usepackage{booktabs}
\usepackage{caption}
\usepackage{natbib}
\setcitestyle{numbers,square}
\usepackage{multirow}
\usepackage{tabularx}

\newcommand{\gf}{ {{\mathbb F}} }
\newtheorem{lemma}{Lemma}[section]
\newtheorem{theorem}{Theorem}[section]
\newtheorem{remark}[theorem]{Remark}

\journal{Finite Field and Their Applications}

\begin{document}

\begin{frontmatter}



\title{The compositional inverses of three classes of permutation polynomials over finite fields}


\author[wuaddress]{Danyao Wu\corref{mycorrespondingauthor}}
\cortext[mycorrespondingauthor]{Corresponding author}
\ead{wudanyao111@163.com}

\author[yuanaddress]{Pingzhi Yuan}
\ead{yuanpz@scnu.edu.cn}
\author[guanaddress]{Huanhuan Guan} 
\ead{guan1110h@163.com}

\author[liaddress]{Juan Li}
\ead{41170208@qq.com}

\address[wuaddress]{School of Computer Science and Technology,
	Dongguan University of Technology, Dongguan 523808, China}
\address[yuanaddress]{School of Mathematics, South China Normal University, Guangzhou 510631, China}
\address[guanaddress]{School of Mathematics and Statistics, Guizhou University of Finance and Economics, Guiyang 550025, China}
\address[liaddress]{School of Mathematics, Jiaying University, Meizhou, 514015, China}

\begin{abstract}
R. Gupta, P. Gahlyan and R.K. Sharma presented three classes of permutation trinomials over $\gf_{q^3}$ in Finite Fields and Their Applications.  In this paper, we  employ the local method  to prove that those polynomials  are indeed permutation polynomials and provide their compositional inverses.
\end{abstract}




\end{frontmatter}


\section{Introduction}
\label{}
Let  $\gf_q$ be the finite field with $q$ elements,  where $q$ is a prime power, and
let $\gf_q[x]$
be the ring of polynomials in a single indeterminate $x$ over $\gf_q$. A polynomial
$f \in\gf_q[x]$ is called a {\em permutation polynomial} (PP) of $\gf_q$ if its
associated polynomial mapping $f: c\mapsto f(c)$ from $\gf_q$ to itself is a bijection. The unique polynomial denoted by $f^{-1}(x)$ over $\gf_q$
such that $f(f^{-1}(x))\equiv f^{-1}(f(x)) \equiv x \pmod{x^q-x}$ is called the compositional inverse of $f(x).$ Furthermore,  $f(x)$ is called  an involution when $f^{-1}(x)=f(x).$

Permutation polynomials over finite fields have been a fascinating
subject of study for many years, with applications in coding
theory \cite{DH13,LC07}, cryptography \cite{RSA,SH}, combinatorial
designs \cite{DY06}, and various other fields
of mathematics and engineering. Detailed information regarding the properties,
constructions, and applications of permutation polynomials can be
found in the works of  Lidl and Niederreiter \cite{LR97}, as well as  Mullen \cite{Mull}. Therefore, discovering new PPs is of great interest in both theoretical and practical aspects.
Some recent advancement in PPs can be found in \cite{AGW,CC17,DY06,Hou151,LWLLZ18,TZJ15,WY22,XLC22,YD,YZ,yuan2022permutation}.

Finding explicit compositional inverse of a random permutation polynomial is a particularly challenging task in determining permutation polynomials.
Fortunately, one of the co-authors presented the local method in \cite{yuan2022local} for determining permutation polynomials and their compositional inverses. 
The work of this coauthor \cite{yuan2022local} inspires us to  seek the compositional inverses of some known permutation polynomials. 

Recently,  R. Gupta, P. Gahlyan and R.K.Sharma \cite{Gupta2022new} presented three new classes of permutation trinomials over $\gf_{q^3}$ by the multivariate method and the resultant of two polynomials. Actually, they had  
derived the compositional inverses of those classes of permutation polynomials by factoring those resultants, but they did not explicitly list the formulars of the inverse permutations.
Computing those resultants  and factoring them is a challenging task.  In this paper, we employ the local method  \cite{yuan2022local} to prove that those polynomials \cite{Gupta2022new} are indeed permutation polynomials and provide their compositional inverses.

The remainder of this paper is organized as follows. In Section 2, we provide a brief overview of the local method, which we use throughout the rest of the paper. In Section 3, we prove three classes of permutation polynomials \cite{Gupta2022new} are indeed permutation polynomials over finite fields  and provide their compositional inverses  by the local method. 

\section{The main Lemma}
In this section, we present the local method about how to find a permutation  polynomial and to compute its compositional inverse over $\gf_{q^n}$ simultaneously. For the convenience of the readers, we provide the proof here again.

\begin{lemma}\label{ff-}\cite[Theorem 2.2]{yuan2022local}
	Let $q$ be a prime power and $f(x)$ be a polynomial over $\gf_q.$ Then
	$f(x)$ is a permutation polynomial if and only if there exist nonempty 
	subsets $S_i$, $i=1, 2, \cdots, t$ of $\gf_q$ and maps $\psi_i: \gf_q \rightarrow S_i$, $i=1, 2, \cdots, t$ such that $\psi_i\circ f(x)=\varphi_i,$ $i=1, 2, \cdots, t$
	and $x=F(\varphi_1, \varphi_2, \cdots, \varphi_t),$ where $F(x_1, x_2, \cdots, x_t)\in \gf_q[x_1, x_2, \cdots, x_t].$ Moreover, the compositional inverse of $f$ is given by
	$$f^{-1}(x)=F(\psi_1, \psi_2, \cdots, \psi_t)$$
\end{lemma}
\begin{proof}
	If there exist nonempty  subsets  $S_i$, $i=1, 2, \cdots, t$ of $\gf_q$ and maps $\psi_i(x): \gf_q \rightarrow S_i$, $i=1, 2, \cdots, t$ such that $\psi_i(x)\circ f(x)=\varphi_i(x),$ $i=1, 2, \cdots, t$
	and $x=F(\varphi_1(x), \varphi_2(x), \cdots, \varphi_t(x)),$ then 
	$$F\left(\psi_1(f(x)), \psi_2(f(x)), \cdots, \psi_t(f(x))\right)=x.$$
	This implies 
	$$f^{-1}(x)=F\left(\psi_1(x), \psi_2(x), \cdots, \psi_t(x)\right).$$
	If $f(x)$ is a permutation polynomial over $\gf_q$, we can take $S=\gf_q,$ $i=1, $ $\psi(x)=f^{-1}(x)$, $\varphi(x)=x$ and $F(x)=x,$ then all conditions are satisfied, and we have $f^{-1}(x)=\psi(x).$ This completes the proof. 
\end{proof}
In order to derive the compositional inverses of the permutations studied by  R. Gupta et al. in  \cite{Gupta2022new}, we introduce  the resultant of two polynomials.

Let $f(x)=a_0x^n+a_1x^{n-1}+\cdots+a_n\in \gf_{q}[x]$ and $g(x)=b_0x^m+b_1x^{m-1}+\cdots+b_m\in \gf_{q}[x]$ be two polynomials of degree $n$ and $m$, respectively, with $n, m \in \mathbb{N}.$ Then the resultant $R(f, g, x)$ of the two polynomials with respect to $x$  is defined by the determinant

\begin{equation*}
	R(f, g, x)=\left|\begin{matrix}
		a_0& a_1&\cdots&a_n& 0& 0&\cdots&0\\
		0 & a_0& a_1&\cdots&a_n& 0&\cdots&0\\
		\vdots&&&&&&&\vdots\\
		0&\cdots&\cdots&\cdots&a_0& a_1&\cdots&a_n\\
		b_0& b_1&\cdots&b_m& 0& 0&\cdots&0\\
		0 & b_0& b_1&\cdots&b_m& 0&\cdots&0\\
		\vdots&&&&&&&\vdots\\
		0&\cdots&\cdots&\cdots&b_0& b_1&\cdots&b_m
	\end{matrix}\right|
\end{equation*}
of order $m+n.$ 

If $f(x)=a_0(x-\alpha_1)(x-\alpha_2)\cdots(x-\alpha_n),$ where $a_0\neq0, $ in the splitting field of $f$ over $\gf_q,$ then $R(f, g, x)$ is also given by 
the formula 
$$R(f, g, x)=a_0^m\prod_{i=1}^{n}g(\alpha_i).$$
Then $R(f, g, x)=0$ if and only if $f$ and $g$ have a common divisor in $\gf_{q}$ of positive degree.

\section{ Permutation trinomials and their compositional inverses over $\gf_{q^{3}}$}
\begin{theorem}\label{th21}
	For a positive integer $m$, let $q=2^m$ and $A \in \gf_{q}^*$ with $A^3=1.$    Then  the polynomial $f_1(x)=x+A x^{q^2-q+1}+x^{q^2+q-1}$ is a permutation polynomial over  $\gf_{q^3}$ if and only if $m\not\equiv 2 \pmod{3}.$ Moreover, if $f_1(x)$ is a permutation polynomial over $\gf_{q^3}$, then the compositional inverse of $f_1(x)$ is 
	
	{\footnotesize
		\begin{equation*}
			f_1^{-1}(x)=\begin{cases}
				\frac{(A x^{q^2+1}+x^{2q^2}+A x^{2q})^{q+1}(x+A x^q+ A^2 x^{q^2})}{(x+A x^q+ A^2 x^{q^2})\circ (A x^{q^2+1}+x^{2q^2}+A x^{2q})^{q+1} }, \,\, & \text{if \,\,$x+A x^q+A^2x^{q^2}\neq0;$}\\
				\frac{x^{q^2+2q+1}}{x^{q^2+2q}+A x^{2q+1}+x^{2q^2+1}},  \,\, &\text{if \,\,$x+A x^q+A^2x^{q^2}=0$\, and \, $x\neq 0;$}\\
				0, \,\, &\text{if \,\,$x=0$.}\\
			\end{cases}
	\end{equation*}}

\end{theorem}
\begin{proof}As it has been shown in \cite{Gupta2022new}  that if $m\equiv 2 \pmod{3},$ then $f_1(x)$ is not a permutation polynomial over $\gf_{q^3}$ and  $f_1(x)$ has a unique  root $0$ in $\gf_{q^3}.$  Now, we only need to show  that if $m \not \equiv 2 \pmod{3}, $  then $f_1(x)$ permutes $\gf_{q^3}^*.$
	
	Let $\psi_1(x)=x, \psi_2(x)=x^q, \psi_3(x)=x^{q^2}, $ $ \varphi_1(x)=\psi_1(x)\circ f_1(x)=f_1(x), \varphi_2(x)=\psi_2(x)\circ f_1(x)=f_1^q(x)$, and $  \varphi_3(x)=\psi_3(x)\circ f_1(x)=f_1^{q^2}(x).$ For simplicity, put $\varphi_1(x)=a, \varphi_2(x)=b, \varphi_3(x)=c.$
	Then  $c=b^q=a^{q^2}.$ We  assume that $x\neq 0$  so that  $abc\neq0.$ By substituting $y=x^q,$ $z=y^q, $ we obtain the system of equations 
	\begin{equation}\label{eq2}
		\begin{cases}
			x+\frac{A xz}{y}+\frac{yz}{x}=a,\\
			y+\frac{A xy}{z}+\frac{xz}{y}=b,\\
			z+\frac{A yz}{x}+\frac{xy}{z}=c,\\
		\end{cases}
	\end{equation}
	which can be rewritten as 
	\begin{equation}\label{eq1}
		\begin{cases}
			x^2y+A x^2z+y^2z=axy,\\
			y^2z+A y^2x+z^2x=byz,\\
			z^2x+A z^2y+x^2y=czx.\\	
		\end{cases}
	\end{equation}
	Since $A^3=1$, it follows from \eqref{eq2} that 
	\begin{equation}\label{abc}
		x+A y+A ^2 z=a+A b+A^2 c.
	\end{equation} 
	By adding the first two equations of \eqref{eq1}, along with  Eq.\,\eqref{abc},  we get 
	\begin{align}\label{eq4}
		0=&\, x^2y+A x^2z+axy	+A y^2x+z^2x+byz\nonumber \\
		=&\, xy(x+A y+a)+xz(A x+z) +byz \nonumber\\
		=&\, xy(A b+A^2c+A^2z)+xz(A a+A^2b+c+A ^2y)+byz \nonumber\\
		=&\,(A b+A^2c)xy+(A a+A^2b+c)xz+byz.		
	\end{align}
	By raising both sides of Eq.\,\eqref{eq4}  to the $q$-th power, we obtain  
	\begin{equation}\label{eq41}
		(A b+A^2c)^qyz+(A a+A^2b+c)^qxy+cxz=0.
	\end{equation}
	Eliminating $yz$ from the system composed of Eqs.\,\eqref{eq4} and \eqref{eq41},
	we obtain
	$$(A ac+c^2+A b^2)xy=(A ac+c^2+A b^2)^qxz,$$ or
	\begin{equation}\label{eq5}
		(A ac+c^2+A b^2)y=(A ac+c^2+A b^2)^qz
	\end{equation} because  $x\neq 0.$
	
	We now prove  $A ac+c^2+A b^2\neq 0.$ Indeed, suppose otherwise. Then
	\begin{equation}\label{eq9}
		A ac+c^2+A b^2= 0.
	\end{equation}	
	By raising both sides of \eqref{eq9} to the $q$-th power and $q^2$-th power, respectively, together with \eqref{eq9}, 
	we have the following system of equations
	\begin{equation}\label{eq14}
		\begin{cases}
			A ac+c^2+A b^2=0,\\
			A ab+a^2+A c^2=0,\\
			A bc+b^2+A a^2=0.\\
		\end{cases}
	\end{equation}
	This implies that 
	\begin{equation*}
		\begin{cases}
			A bc+b^2+A a^2=0,\\
			A ^2 ac+A ab+bc=0,
		\end{cases}
	\end{equation*}
	or \begin{equation}\label{eq10}
		\begin{cases}
			b(A c+b)	+A a^2=0,\\
			A a (A c+b)+bc=0,
		\end{cases}
	\end{equation}
	as $A^3=1.$
	Since $abc\neq 0$, by \eqref{eq10},  we have 
	$	A a^2/b=bc/(A a), $ or 
	\begin{equation}\label{eq12}
		A^2= a^{q^2+2q-3}
	\end{equation} as $b=a^q, c=a^{q^2}.$
	Put $a^{q-1}=u.$ Then  $u^{q^2+q+1}=1$  and the above equation reduces to
	\begin{equation}\label{eq13}
		A ^2= u^{q+3}.
	\end{equation}
	By raising both sides of Eq.\,\eqref{eq13} to the third power and using $A^3=1, $ we get  $u^{3(q+3)}=1,$ which leads to
	\begin{equation}\label{u1} u^{3(2q-1)}=u^{q(3(q+3))-3(q^2+q+1)}=1.
		\end{equation}  
	If $m \not\equiv 2 \pmod{3},$ then $2q\not\equiv1 \pmod{7}.$
	This implies that 
	\begin{align*}
		\gcd(2q-1, q^3-1)=&\,\gcd(2q-1, 8q^3-8)\\
		=&\,\gcd(2q-1, ((2q-1)+1)^3-8)\\
		=&\,\gcd(2q-1, 7)\\
		=&\,1,		
	\end{align*}
where the first step holds since $\gcd(2q-1, 8)=1.$
	Therefore, there exist two integers $k, l$ such that $k(2q-1)+l(q^3-1)=1.$
We deduce that 
	$$u^3=u^{3\left(k(2q-1)+l(q^3-1)\right)}=\left(u^{3(2q-1)}\right)^k=1$$
	According to  Eq. \eqref{eq13},  $A^2=u^{q+3}=u^q.$ If $A=1,$ then $u^q=1=u^3.$  Since $\gcd(q,3)=1,$ $a^{q-1}=u=u^{\gcd(q, 3)}=1.$ So $b=a^q=a$ and $c=b^q=a^q=b.$ These together with Eq. \eqref{eq9} imply that $a^2=0,$ which is a contradiction. If $A$ is primitive in $\gf_{4},$
	then by Eq. \eqref{eq13}, $u^q=A^2=A+1\neq1.$ So $u\neq1. $ Since $u^3=1,$ we know that $u^2+u+1=0.
$ Since $u=a^{q-1}=b/a,$ we have that $(b/a)^2+(b/a)+1=0,$ i.e., that $b^2+ab+a^2=0.$ Raising the last equation to the $q-$th and $q^2-$th powers, we obtain that $c^2+bc+b^2=0$ and that $a^2+ca+c^2=0.$ Adding up all three equations yields that $ab+bc+ca=0.$ Thus, adding up all three equations in Eq.\eqref{eq14}, we have that $(a^2+b^2+c^2)(A+1)=0.$ Since $A+1=A^2\neq0, $ $a^2+b^2+c^2=0.$ So $ab=a^2+b^2=c^2.$ By the second equation in Eq. \eqref{eq14}, $a^2=0,$ which is impossible.
 We have shown $ A ac+c^2+A b^2\neq0$ completely.
	
	Raising both sides of Eq.\,\eqref{eq5} to $q$-th power, we have 
	\begin{equation}\label{eq6}
		(A ac+c^2+A b^2)y=(A ac+c^2+A b^2)^qz=(A ac+c^2+A b^2)^{q^2}x.
	\end{equation}
	It follows from Eqs.\,\eqref{abc} and \eqref{eq6} that 
	\begin{align} \label{th1x}
		&\left(	(A ac+c^2+A b^2)^{q+1}+A	(A ac+c^2+A b^2)^{q^2+q}+ A^2	(A ac+c^2+A b^2)^{q^2+1}\right)x \nonumber\\
		=&\,	(A ac+c^2+A b^2)^{q+1}(a+A b+A^2c).
	\end{align}
	
	If $	(A ac+c^2+A b^2)^{q+1}+A	(A ac+c^2+A b^2)^{q^2+q}+ A^2	(A ac+c^2+A b^2)^{q^2+1}\neq 0$, then
	\begin{align*}
		x=&\,\frac{(A ac+c^2+A b^2)^{q+1}(a+A b+A^2c)}{	(A ac+c^2+A b^2)^{q+1}+A 	(A ac+c^2+A b^2)^{q^2+q}+ A^2	(A ac+c^2+A b^2)^{q^2+1}}.
	\end{align*}
	
	If $	(A ac+c^2+A b^2)^{q+1}+A	(A ac+c^2+A b^2)^{q^2+q}+ A^2	(A ac+c^2+A b^2)^{q^2+1}= 0$, then  $a+A b+A^2c=0$ by Eq.\,\eqref{th1x},  and the latter condition is equivalent to $x+A y+A^2 z=0$ by Eq.\,\eqref{abc}. 
	By adding the first two equations of \eqref{eq1} , we obtain  
	$$0=ax(A^2x+A z)+b(A^2x+A z)z$$ because  $y=A^2x+Az.$  Furthermore, by   $y=A^2x+Az\neq 0$, we have  $ax=bz,$  and so $by=cx. $  
	Together with the first equation of \eqref{eq2} we get
	$$(b^2c+A ab^2+ac^2)x=ab^2c.$$
	Since  $abc\neq0,$ we  have  
	$b^2c+A ab^2+ac^2\neq0.$ Then  
	\begin{align*}
		x=&\,\frac{ab^2c}{b^2c+A ab^2+ac^2}.
	\end{align*} It follows from Lemma \ref{ff-} that $f_1(x)$ permutes $\gf_{q^3}$ and the 
	compositional inverse of $f_1(x)$ is 
	{\footnotesize
		\begin{equation*}
			f_1^{-1}(x)=\begin{cases}
				\frac{(A x^{q^2+1}+x^{2q^2}+A x^{2q})^{q+1}(x+A x^q+ A^2 x^{q^2})}{(x+A x^q+ A^2 x^{q^2})\circ (A x^{q^2+1}+x^{2q^2}+A x^{2q})^{q+1} }, \,\, & \text{if \,\,$x+A x^q+A^2x^{q^2}\neq0;$}\\
				\frac{x^{q^2+2q+1}}{x^{q^2+2q}+A x^{2q+1}+x^{2q^2+1}},  \,\, &\text{if \,\,$x+A x^q+A^2x^{q^2}=0$\,\, and \,\, $x\neq 0;$}\\
				0, \,\, &\text{if \,\,$x=0$},\\
			\end{cases}
	\end{equation*}}
which is the desired result. 

\end{proof}

\begin{remark}
	R. Gupta, P. Gahlyan and R.K.Sharma  \cite{Gupta2022new} investigated this  class of permutation trinomials over $\gf_{2^{3m}}$ in Theorem 3.2. Actually, by eliminating the indeterminate $z$ using \eqref{eq1}, they obtained 
	{\footnotesize \begin{align*}
			f(y)&:=\, Ay^4+(a+x)y^3+(bx+ab)y^2+(aAx^2+Ax^3)y+Aabx^2+Abx^3+a^2x^2=\,0,\\
			g(y)&:=\, y^4+(a^2A+Ax^2+ac+cx)y^2+(x^3+a^2x)y+A^2x^4+cAx^3+acAx^2=\,0.
	\end{align*}}
	The resultant of $f$ and $g$ with respect to $y$ was factored as 
	$$R(f, g, x):= x^4(x+a)^8(\alpha x+\beta),$$
	where 
	\begin{align*}
		\alpha&=\,ab^2A^2+bc^2A^2+ca^2A^2+a^3+b^3+c^3+abc,\\
		\beta&=\, a^2b^2A^2+a^3cA^2+abc^2A^2+c^4A^2+a^2c^2A+ab^3+a^2bc+b^2c^2+ac^3.
	\end{align*}
	By $R(f, g, x):= x^4(x+a)^8(\alpha x+\beta)=0$, they got 
	\begin{equation*}
		x=
		\begin{cases}
			a, &\text{if}\,\, bA+a=0;\\
			\beta/\alpha, & \text{if }\,\, bA+a\neq0.
		\end{cases}
	\end{equation*}
	Note that if  $ bA+a=0$, then  
	$x=\beta/\alpha=a$ still holds, and so for any $x\in \gf_{q^3}^*,$ we have $x=\beta/\alpha.$   
	Hence, by Lemma \ref{ff-}, we have that the compositional inverse of $f_1(x)$ is {\footnotesize
		\begin{equation*}
			f_1^{-1}(x)=\\
			\begin{cases}\frac{A^2(x^{2q+2}+x^{q^2+3}+x^{2q^2+q+1}+x^{4q^2})+Ax^{2q^2+2}+x^{3q+1}+x^{q^2+q+2}+x^{2q^2+2q}+x^{3q^2+1}}{A^2(x^{2q+1}+x^{2q^2+q}+x^{q^2+2})+x^3+x^{3q}+x^{3q^2}+x^{q^2+q+1}}, &\text {if}\,\, x\neq0;\\
				0,\,\, &\text {if}\,\, x=0.\
			\end{cases}
		\end{equation*}
	}
\end{remark}

\begin{theorem}
	For a positive integer $m$, let $q=2^m$ and $A \in \gf_{q}^*$ with $A^3=1.$  Then the polynomial $f_2(x)=x+A x^{q^3-q^2+q}+x^{q^2+q-1}$ is a permutation polynomial over  $\gf_{q^3}$ if and only if $m\not\equiv 1 \pmod{3}.$ Moreover, if $f_2(x)$ permutes $\gf_{q^3}$, then 
	the compositional inverse of $f_2(x)$ is {\footnotesize
		\begin{equation*}
			f_2^{-1}(x)=\begin{cases}
				\frac{(x^2+A x^{q^2+1}+A x^{2q})^{q+1}(x+A^2x^q+A x^{q^2})}{(x+A^2x^q+A x^{q^2})\circ (x^2+A x^{q^2+1}+A x^{2q})^{q+1}}, \,\, & \text{if\,\,$x+A^2 x^q+A x^{q^2}\neq0;$}\\
				\frac{x^{q^2+2q+1}}{x^{q^2+2q}+A x^{2q^2+1}+x^{2q+1}},  \,\, &\text{if\,\,$x+A ^2x^q+A x^{q^2}=0$\,\, and \,\, $x\neq 0;$}\\
				0, \,\, &\text{if\,\,$x=0$}.\\
			\end{cases}
	\end{equation*}}
\end{theorem}
\begin{proof}As it has been shown in \cite{Gupta2022new}  that if $m\equiv 1 \pmod{3},$ then $f_2(x)$ is not a permutation polynomial over $\gf_{q^3}$ and  $f_2(x)$ has a unique  root $0$ in $\gf_{q^3}.$  Now, we only need to show  that if $m \not \equiv 1 \pmod{3}, $  then $f_2(x)$ permutes $\gf_{q^3}^*.$
	
	Let $\psi_1(x)=x, \psi_2(x)=x^q, \psi_3(x)=x^{q^2}, $ $ \varphi_1(x)=\psi_1(x)\circ f_2(x)=f_2(x), \varphi_2(x)=\psi_2(x)\circ f_2(x)=f_2^q(x)$, and $  \varphi_3(x)=\psi_3(x)\circ f_2(x)=f_2^{q^2}(x).$ For simplicity, put $\varphi_1(x)=a, \varphi_2(x)=b, \varphi_3(x)=c.$
	Then we have $c=b^q=a^{q^2}.$ We  assume that $x\neq 0$ so that  $abc\neq0.$ By substituting $y=x^q,$ $z=y^q, $ we obtain the system of equations 
	\begin{equation}\label{eq211}
		\begin{cases}
			x+\frac{A xy}{z}+\frac{yz}{x}=a,\\
			y+\frac{A yz}{x}+\frac{xz}{y}=b,\\
			z+\frac{A xz}{y}+\frac{xy}{z}=c,\\
		\end{cases}
	\end{equation}
	which can be rewritten as 
	\begin{equation}\label{eq111}
		\begin{cases}
			x^2z+A x^2y+z^2y=axz,\\
			y^2x+A y^2z+x^2z=byx,\\
			z^2y+A z^2x+y^2x=cyz.\\	
		\end{cases}
	\end{equation}
	Since $A ^3=1,$ it follows from \eqref{eq211} that  
	\begin{equation}
		\label{abc11}
		x+A^2 y+A  z=a+A^2 b+A c.
	\end{equation} 
	By adding the  first two equations of \eqref{eq111}, along with  Eq.\,\eqref{abc11}, we have 
	\begin{align}	\label{eq411}
		0=&\, A x^2y+z^2y+ y^2x+A y^2z+axz+byz \nonumber\\
		=&\, xy(A x+ y+b)+yz(A y+z) +axz \nonumber\\
		=&\, xy(A a +A^2c+A^2z)+yz(A^2 a+A b+c+A ^2x)+axz \nonumber\\
		=&\,(A a+A^2c)xy+(A^2 a+A b+c)xz+axz.
	\end{align}
	By raising both sides of Eq.\,\eqref{eq411} to the $q$-th power, we obtain 
	\begin{equation}\label{eq4111}
		(A a+A^2c)^qyz+(A^2 a+A b+c)^qxz+bxy=0.
	\end{equation}
	Combining Eqs.\,\eqref{eq411} and \eqref{eq4111},  and then eliminating $yz$, 
	we obtain  $$(a^2+A ac+A b^2)xy=(a^2+A ac+A b^2)^qxz,$$ or 
	\begin{equation}\label{eq511}
		(a^2+A ac+A b^2)y=(a^2+A ac+A b^2)^qz,
	\end{equation}
	as $x\neq 0.$ 
	
	We claim that $a^2+A ac+A b^2\neq 0.$	Suppose, on the contrary,  that 
	\begin{equation}
		\label{eq911}
		a^2+A ac+A b^2= 0.
	\end{equation}	
	Raising both sides of Eq.\,\eqref{eq911} to the $q$-th power and $q^2$-th power, respectively, together with Eq.\,\eqref{eq911}, 
	we have the system of equations
	\begin{equation}
		\label{eq1411}
		\begin{cases}
			a^2+	A ac+A b^2=0,\\
			b^2+	A ab+A c^2=0,\\
			c^2+A bc+A a^2=0.
		\end{cases}
	\end{equation}
	This implies that 
	\begin{equation*}
		\begin{cases}
			c^2+A bc+A a^2=0,\\
			A ac+A ^2 ab+bc=0,
		\end{cases}
	\end{equation*} or 
	\begin{equation}
		\label{eq1011}
		\begin{cases}
			c( c+A b)	+A a^2=0,\\
			A a (c+A b)+bc=0.
		\end{cases}
	\end{equation}
	Since $abc\neq 0$, we obtain from Eq.\,\eqref{eq1011} that either $	A a^2/c=bc/(A a),$ or
	\begin{equation}
		\label{eq1211}
		A^2= a^{2q^2+q-3}
	\end{equation} as $b=a^q, c=a^{q^2}.$
	Put $a^{q-1}=u.$ Then $u^{q^2+q+1}=1$  and the above equation reduces to 
	\begin{equation}
		\label{eq1311}
		A ^2= u^{2q+3}.
	\end{equation}
	By raising both sides of Eq.\,\eqref{eq1311} to the third power and using $A^3=1$,  we obtain  $u^{3(2q+3)}=1$, 	
	which leads to $u^{3(q-2)}=u^{q(3(2q+3))-6(q^2+q+1)}=1.$ 
	Moreover, we have that $$\gcd(q-2, q^3-1)=\gcd(q-2, \left((q-2)+2\right)^3-1)=\gcd(2^m-2, 7)=1,$$
	where the last step holds since $2^m\not\equiv 2 \pmod{7}$ when $m\not \equiv 1 \pmod{3}.$ Thus, $k(q-2)+l(q^3-1)=1$ for some $k, l \in \mathbb{Z}$. So we deduce that 
	$$u^3=u^{3\left(k(q-2)+l(q^3-1)\right)}=\left(u^{3(q-2)}\right)^k=1.$$	

Since the remaining proof is similar to Theorem \ref{th21}, we only give the results here. 

	It follows from Lemma \ref{ff-} that 
	$f_2(x)$ is a permutation polynomial over $\gf_{q^3}$ and  the compositional inverse of $f_2(x)$ is 
	{\footnotesize
		\begin{equation*}
			f_2^{-1}(x)=\begin{cases}
				\frac{(x^2+A x^{q^2+1}+A x^{2q})^{q+1}(x+A^2x^q+A x^{q^2})}{(x+A^2x^q+A x^{q^2})\circ (x^2+A x^{q^2+1}+A x^{2q})^{q+1}}, \,\, & \text{if\,\, $x+A^2 x^q+A x^{q^2}\neq0;$}\\
				\frac{x^{q^2+2q+1}}{x^{q^2+2q}+A x^{2q^2+1}+x^{2q+1}},  \,\, &\text{if\,\, $x+A ^2x^q+A x^{q^2}=0$\,\, and \,\, $x\neq 0;$}\\
				0, \,\, &\text{if\,\,$x=0$}.\\
			\end{cases}
	\end{equation*}}
	This completes the proof.
	
\end{proof}

\begin{remark}
	R. Gupta, P. Gahlyan and R.K.Sharma  \cite{Gupta2022new} investigated this  class of permutation trinomials over $\gf_{2^{3m}}$ in Theorem 3.4. Actually, by eliminating the indeterminate $z$ using \eqref{eq111}, they obtained 
	{\footnotesize \begin{align*}
			f(x)&:=\, Ax^4+(b+y)x^3+(ay+ab)x^2+(bAy^2+Ay^3)x+Aaby^2+Aay^3+b^2y^2=\,0,\\
			g(x)&:=\, x^4+(b^2A+Ay^2+bc+cy)x^2+(y^3+b^2y)x+A^2y^4+cAy^3+bcAy^2=\,0.
	\end{align*}}
	The resultant of $f$ and $g$ with respect to $x$ was factored as 
	$$R(f, g, x):= y^4(y+b)^8(\alpha y+\beta),$$
	where 
	\begin{align*}
		\alpha&=\,ba^2A^2+b^2cA^2+c^2aA^2+a^3+b^3+c^3+abc,\\
		\beta&=\, a^2b^2A^2+b^3cA^2+abc^2A^2+c^4A^2+b^2c^2A+ba^3+b^2ac+a^2c^2+bc^3.
	\end{align*}
	By $R(f, g, x):= y^4(y+b)^8(\alpha y+\beta)=0$, they got 
	\begin{equation*}
		y=
		\begin{cases}
			b, &\text{if}\,\, cA+a=0;\\
			\beta/\alpha, & \text{if }\,\, cA+a\neq0.
		\end{cases}
	\end{equation*}
	Note that if  $ bA+a=0$, then 
	$y=b=\beta/\alpha $ still holds,  and so for  any $y\in \gf_{q^3}^*,$ we have $y=\beta/\alpha. $  This implies $x=\beta^{q^2}/\alpha^{q^2}$, as $y=x^q. $ Moreover, since $c=b^q=a^{q^2},$ we have 
	{\footnotesize
		\begin{align*}
			x=&\, \frac{\left(a^2b^2A^2+b^3cA^2+abc^2A^2+c^4A^2+b^2c^2A+ba^3+b^2ac+a^2c^2+bc^3\right)^{q^2}}{\left(ba^2A^2+b^2cA^2+c^2aA^2+a^3+b^3+c^3+abc\right)^{q^2}}\\
			=&\, \frac{a^2c^2A^2+a^3bA^2+ab^2cA^2+b^4A^2+a^2b^2A+ac^3+a^2bc+b^2c^2+ab^3)}{(ba^2A^2+b^2cA^2+c^2aA^2+a^3+b^3+c^3+abc)}.
		\end{align*}
	}
	Hence, by Lemma \ref{ff-}, we have that  the compositional inverse of $f_2(x)$ is {\footnotesize
		\begin{equation*}
			f_2^{-1}(x)=\\
			\begin{cases}  \frac{A^2(x^{2q^2+2}+x^{q+3}+x^{q^2+2q+1}+x^{4q})+Ax^{2q+2}+x^{3q^2+1}+x^{q^2+q+2}+x^{2q^2+2q}+x^{3q+1}}{A^2(x^{2+q}+x^{2q+q^2}+x^{2q^2+1})+x^3+x^{3q}+x^{3q^2}+x^{q^2+q+1}}, \,\, &\text{if}\,\, x\neq0;\\
				0, \,\, &\text{if}\,\, x=0.		
			\end{cases}
	\end{equation*}}
\end{remark}

\begin{theorem}\label{thm3}
	Let $q$ be a prime power and $A \in \gf_{q}*$. Then the polynomial 
	$$f_3(x)=x+A x^{q^2-q+1}+A^2x^{q^2}$$ is a  permutation polynomial  over $\gf_{q^3}$ if and only if $A^3\neq1.$ Moreover, if $f_3(x)$ permutes $\gf_{q^3}$, the compositional inverse of $f_3(x)$ is 
\begin{equation*}
	f_3^{-1}(x)=
	(A^2x+x^q+Ax^{q^2})^{q^3-2}x^{q+1}.
\end{equation*}
	
\end{theorem}

\begin{proof}
	Let $\psi_1(x)=x, \psi_2(x)=x^q, \psi_3(x)=x^{q^2}, $ $ \varphi_1(x)=\psi_1(x)\circ f_3(x)=f_3(x), \varphi_2(x)=\psi_2(x)\circ f_3(x)=f_3^q(x)$, and $  \varphi_3(x)=\psi_3(x)\circ f_3(x)=f_3^{q^2}(x).$ For simplicity, put $\varphi_1(x)=a, \varphi_2(x)=b, \varphi_3(x)=c.$
	Then we have $c=b^q=a^{q^2}.$
	We assume that  $A^3-1=(A-1)(A^2+A+1)=0$.
	If $A^2+A+1=0$, then $f_3(1)=f_3(0)=0.$ If $A=1,$ then for any $\alpha \in\{x \in \gf_{q^3}: x^{q+1}+x^q+1=0\}$, we have $\alpha^{q^2+q+1}=1$.
	Therefore,  there exists $\beta\in \gf_{q^3}^*$ such that $\alpha=\beta^{q-1}.$
	This implies 
	$$f_3(\beta)=\beta+\beta^{q^2-q+1}+\beta^{q^2}=\beta(1+\beta^{q^2-q}+\beta^{q^2-1})=\beta(1+\alpha^q+\alpha^{q+1})=0.$$
	Hence, $f_3(x) $ is not a permutation polynomial over $\gf_{q^3}.$ 
	
	  Now  we  assume that $A^3\neq1.$  
	For any $x\in \gf_{q^3}^*$, by substituting $y=x^q,$ $z=y^q$, we obtain the system of equations
	\begin{equation*}
		\begin{cases}
			x+A\frac{xz}{y}+A^2z&=a,\\
			y+A\frac{xy}{z}+A^2x&=b,\\
			z+A\frac{yz}{x}+A^2y&=c,\\
		\end{cases}
	\end{equation*}
	or, we can rewrite as 
	\begin{equation}\label{th3eq1}
		\begin{cases}
			xy+Axz+A^2yz&=ay,\\
			yz+Axy+A^2xz&=bz,\\
			xz+Ayz+A^2xy&=cz.\\
		\end{cases}
	\end{equation}
	By the first two equations of  \eqref{th3eq1}, we have 
	\begin{equation}\label{th3eq2}
		yz-A^3yz=bz-aAy.
	\end{equation}
	Since $x\in \gf_{q^3}^*$, we get $yz\neq0,$ and so by Eq. \eqref{th3eq2}, we obtain 
	\begin{equation}\label{th3eq3}
		1-A^3=\frac{b}{y}-\frac{Aa}{z}.
	\end{equation}
	By raising both sides of Eq.\,\eqref{th3eq3} to the $q$-th power  and $q^2$-th power, respectively, we have the system of equations 
	\begin{equation}\label{th3eq9}
		\begin{cases}
			1-A^3&=\frac{b}{y}-\frac{Aa}{z},\\
			1-A^3&=\frac{c}{z}-\frac{Ab}{x},\\
			1-A^3&=\frac{a}{x}-\frac{Ac}{y}.\\
		\end{cases}
	\end{equation}
	Moreover,  
	$1-A^3\neq0$ implies that for any $x\in \gf_{q^3}^*$, we have $abc\neq0.$
	
	Suppose, on the contrary, that if $1-A^3\neq0,$ there exists $x_0\in \gf_{q^3}^*$ such that $a=0.$ Then we have $x_0+Ax_0^{q^2-q+1}+A^2x_0^{q^2}=0,$ or 
	\begin{equation}\label{th3eq4}
		x_0^{q+1}+Ax_0^{q^2+1}+A^2x_0^{q^2+q}=0.
	\end{equation}
	By raising both sides of Eq.\,\eqref{th3eq4} to $q$-th power, we have 
	\begin{equation}\label{th3eq5}
		x_0^{q^2+q}+Ax_0^{1+q}+A^2x_0^{q^2+1}=0.
	\end{equation} It follows from Eqs.\,\eqref{th3eq4} and \eqref{th3eq5} that 
	$(1-A^3)x_0^{q^2+q}=0,$ and so $1-A^3=0,$ which is a contraction. 
	
	Considering the system  \eqref{th3eq9} as a linear system with variables $(1/x, 1/y, 1/z)$, the determinant of the coefficient matrix of \eqref{th3eq9} is non-zero. Then the system \eqref{th3eq9} has a unique solution over $\gf_{q^3}$ and we have 
	$$1/x=(acA^2+c^2A+bc)(1-A^3)\cdot\left((1-A^3)abc\right)^{-1},$$ or 
	\begin{align} \label{eq3thm3}
		x&=\, (aA^2+cA+b)^{-1}ab.
	\end{align} 
Since $f(0)=0$, it follows from Lemma \ref{ff-} that $f_3(x)$ permutes $\gf_{q^3}$ and  the compositional inverse of $f_3(x)$ is 
	\begin{equation*}
		f_3^{-1}(x)=
				(A^2x+x^q+Ax^{q^2})^{q^3-2}x^{q+1}.
	\end{equation*}
	We are done.
\end{proof}

\begin{remark}
	
Despite the proof of Theorem \ref{thm3}  having been previously presented in Theorem 1 of \cite{wu2024permutation}, we provide the proof here as well. This is due to the fact that,  
regardless of whether  $A^3$ equals $1$ or not, 
 the equation $x(A^2a+b+Ac)=ab$ (or $ x(A^2\varphi_1(x)+\varphi_2(x)+A\varphi_3(x)=\varphi_1(x)\varphi_2(x)$  in Theorem 1 of \cite{wu2024permutation} )
holds.
 However,  Theorem \ref{thm3} calculates the unique solution of the system \eqref{th3eq9} 
  under the condition that 
  $A^3 \neq 1$ to obtain \eqref{eq3thm3}.

 Additionally, R. Gupta, P. Gahlyan and R.K.Sharma  \cite{Gupta2022new} investigated this  class of permutation trinomials over $\gf_{q^{3m}}$ in Theorem 4.1, where $q$ is an odd prime power. We can see that if $q$ is a power of $2$, the result still holds.  
	
	Moreover,  R. Gupta, P. Gahlyan and R.K.Sharma  \cite{Gupta2022new} derived the compositional inverse of permutation polynomial $f_3(x)$ over $\gf_{q^3}$. By eliminating the indeterminate $z$ using \eqref{th3eq1}, they obtained 
	{\footnotesize \begin{align*}
			f(y)&:=\, (Ax^2+A^2bx-A^4x^2+cx)y+A^2cx^2-bcx=\,0,\\
			g(y)&:=\, (A^3x-x+a)y^2+(xb-ab+A^2ax)y=\,0.
	\end{align*}}
	The resultant of $f$ and $g$ with respect to $y$ was factored as 
	$$R(f, g, x):=Acx^2(A^2x-b)\left((A^3-1)x-bA\right)\left((aA^2+cA+b)x-ab\right).$$
	By $R(f, g, x):=Acx^2(A^2x-b)\left((A^3-1)x-bA\right)\left((aA^2+cA+b)x-ab\right)=0$, they got 
	\begin{equation*}
		x=
		\begin{cases}
			a, &\text{if}\,\,  a\in \gf_q^* \,\, \text{and} \,\,  A=-1;\\
			(A^2+A+1)^{-1}a, & \text{if }\,\,  a\in \gf_q^* \,\, \text{and} \,\,  A\neq-1;\\
			A^{-2}b, &\text{if}\,\,  a\in \gf_{q^3}^*\setminus\gf_{q}^* \,\, \text{and} \,\,  a+bA=0;\\
			(aA^2+cA+b)^{-1}ab, &\text{if}\,\,  a\in \gf_{q^3}^*\setminus\gf_{q}^* \,\, \text{and} \,\,  a+bA\neq0.
		\end{cases}
	\end{equation*}
	
	Note that if   $a\in \gf_q^* $ and $ A=-1,$ we  have $a=b=c, $ and so  $x=a=(aA^2+cA+b)^{-1}ab.$ 	
	If $a\in \gf_q^* $ and $ A\neq-1,$ we   have $a=b=c, $ and so $x=(A^2+A+1)^{-1}a=(aA^2+cA+b)^{-1}ab.$

	If $ a\in \gf_{q^3}^*\setminus\gf_{q}^*$ and $a+bA=0$, then  $b+cA=0,$ and so $x=	A^{-2}b=(aA^2+cA+b)^{-1}ab.$
	Hence,  for any $x\in \gf_{q^3}^*,$ $x=(aA^2+cA+b)^{-1}ab$ holds. 
	
Since $f(0)=0$, it follows from Lemma \ref{ff-} that  the compositional inverse of $f_3(x)$ is 
		\begin{equation*}
		f_3^{-1}(x)=
		(A^2x+x^q+Ax^{q^2})^{q^3-2}x^{q+1}.
	\end{equation*}
\end{remark}

\noindent
{\bf Acknowledgments}
\\

P. Yuan was supported by the National Natural Science Foundation of China (Grant No. 12171163), Guangdong Basic and Applied Basic Research Foundation (Grant No. 2024A1515010589). D. Wu was supported by Guangdong Basic and Applied Basic Research Foundation (Grant No. 2020A1515111090). 
\\

\noindent

\end{document}